\documentclass[11pt]{amsart}
\usepackage[utf8]{inputenc}
\usepackage[T1]{fontenc}
\usepackage{amssymb,graphicx}
\usepackage{xcolor}
\definecolor{darkred}{rgb}{0.5,0,0}
\definecolor{darkgreen}{rgb}{0,0.5,0}
\definecolor{darkblue}{rgb}{0,0,0.5}

\usepackage{textcomp}
\usepackage{tcolorbox}
\usepackage{xargs}
\usepackage[colorinlistoftodos,prependcaption,textsize=normalsize]{todonotes}

\usepackage{tikz}
\usetikzlibrary{calc}

\usepackage[colorlinks,pdftex]{hyperref}
\hypersetup{%
pdfauthor={Bart\l{}omiej Bosek, Jaros\l{}aw Grytczuk, William T. Trotter},%
pdftitle={Local Dimension is Unbounded for Planar Posets},%
pdfsubject={Combinatorics of partially ordered sets},%
pdfkeywords={Boolean dimension, local dimension, planar poset},%
pdfproducer={LaTeX},%
pdfcreator={pdfLaTeX},%
colorlinks,
linkcolor={darkblue},
filecolor={darkgreen},
urlcolor={darkred},
citecolor={darkblue}}
\usepackage{cleveref}

\newtheorem{theorem}{Theorem}[section]

\newtheorem{claim}{Claim}

\newtheorem{proposition}[theorem]{Proposition}

\theoremstyle{remark}

\renewcommand{\le}{\leqslant}
\renewcommand{\ge}{\geqslant}

\newcommand{\bfk}{\mathbf{k}}

\newcommand{\bftwo}{\mathbf{2}}
\newcommand{\cgB}{\mathcal{B}}

\newcommand{\cgL}{\mathcal{L}}
\newcommand{\cgM}{\mathcal{M}}

\newcommand{\cgR}{\mathcal{R}}

\newcommand{\Inc}{\operatorname{Inc}}
\newcommand{\Max}{\operatorname{Max}}
\newcommand{\Min}{\operatorname{Min}}
\newcommand{\ldim}{\operatorname{ldim}}
\newcommand{\bdim}{\operatorname{bdim}}

\newcommand{\ple}{\operatorname{ple}}

\newcommand{\PRam}{\operatorname{PRam}}
\newcommand{\Kelly}{\operatorname{K}}
\newcommand{\Core}{\operatorname{Core}}

\begin{document}

\title[PLANAR POSETS AND LOCAL DIMENSION]{Local Dimension is Unbounded for\\
 Planar Posets}%\addtitlefootnote} 

\author[B. BOSEK]{Bart\l omiej Bosek$^1$}
\thanks{$^1$ Supported by Polish National Science Center grant 2013/11/D/ST6/03100.}
\address{B.{ }Bosek: Theoretical Computer Science Department\\
 Faculty of Mathematics and Computer Science\\
 Jagiellonian University in Krak\'ow\\
 Krak\'ow, Poland}
\email{bosek@tcs.uj.edu.pl}

\author[J. GRYTCZUK]{Jaros{\l}aw Grytczuk$^2$}
\thanks{$^2$ Supported by Polish National Science Center grant 2015/17/B/ST1/02660.}
\address{J.{ }Grytczuk: Faculty of Mathematics and Information Science\\
 Warsaw University of Technology\\
 Warszawa, Poland}
\email{j.grytczuk@mini.pw.edu.pl}

\author[W.T. TROTTER]{William T. Trotter$^3$}
\thanks{$^3$ Supported by a Simons Foundation Collaborative Research Grant.}
\address{W.T.{ }Trotter: School of Mathematics\\
 Georgia Institute of Technology\\
 Atlanta, Georgia 30332}
\email{trotter@math.gatech.edu}

%\date{December 5, 2019}


% 06A07 Order, lattices, ordered algebraic structures - Ordered sets - Combinatorics of partially ordered sets
% 05C35 Combinatorics - Graph theory - Extremal problems

% G.2.1: Combinatorics
% G.2.2: Graph Theory


\begin{abstract}
In 1981, Kelly showed that planar posets can have arbitrarily large
dimension. However, the posets in Kelly's example have bounded
Boolean dimension and bounded local dimension, leading naturally
to the questions as to whether either Boolean dimension or local
dimension is bounded for the class of planar posets. The question
for Boolean dimension was first posed by Ne\v{s}et\v{r}il and 
Pudl\'{a}k in 1989 and remains unanswered today. The concept of 
local dimension is quite new, introduced in 2016 by Ueckerdt. Since that time, 
researchers have obtained many interesting results concerning
Boolean dimension and local dimension, contrasting these
parameters with the classic Dushnik-Miller concept of dimension,
and establishing links between both parameters
and structural graph theory, path-width and tree-width
in particular.

Here we show that the local dimension
is not bounded in the class of planar posets. Our proof also
shows that the local dimension of a poset is not bounded in terms
of the maximum local dimension of its blocks, and it 
provides an alternative proof of the fact that the local dimension 
of a poset cannot be bounded in terms of the tree-width of 
its cover graph, independent of its height. 
\end{abstract} 

\maketitle

\section{Notation, Terminology and Background Discussion}

In this paper, we investigate combinatorial problems for
finite posets. As has become standard in the literature, we
use the terms \textit{elements} and \textit{points} interchangeably
in referring to the members of the ground set of a poset. 
We will write $x\parallel y$ in $P$ when $x$ and $y$ are incomparable
in a poset $P$, and we let $\Inc(P)$ denote the set of all ordered pairs 
$(x,y)$ with $x\parallel y$ in $P$. As a binary 
relation, $\Inc(P)$ is symmetric. Recall that
a non-empty family $\cgR$ of linear extensions
of $P$ is called a \textit{realizer} of $P$ when 
$x<y$ in $P$ if and only if $x<y$ in $L$ for each $L\in\cgR$.
Clearly, a non-empty family $\cgR$ of linear extensions of $P$ is a 
realizer of $P$ if and only if for each $(x,y)\in
\Inc(P)$, there is some $L\in\cgR$ for which $x>y$ in $L$.
The \textit{dimension} of a poset $P$, as defined by Dushnik
and Miller in their seminal paper~\cite{bib:DusMil41}, is the least
positive integer $d$ for which $P$ has a realizer $\cgR$ with
$|\cgR|=d$.

In recent years, researchers have been investigating combinatorial
problems for two variations of the Dushnik-Miller concept for dimension,
known as \textit{Boolean dimension} and \textit{local dimension},
respectively.

Here is the setting  for Boolean dimension.
For a positive integer $d$, we let $\bftwo^d$ denote the set of all
$0$--$1$ strings of length~$d$. Such strings are
also called \textit{bit strings}. Let $P$ be a poset and
let $\cgB=\{L_1,L_2,\dots,L_d\}$ be a family of linear orders on the
ground set of $P$ (these linear orders need not be linear extensions of $P$).
Also, let $\tau$ be a function which maps all $0$--$1$ strings of length~$d$
to $\{0,1\}$. For each pair $(x,y)$ of distinct elements of $P$,
we form the bit string $q(x,y,\cgB)$ which has value $1$ in coordinate $i$
if and only if $x<y$ in $L_i$. We call the pair $(\cgB,\tau)$ a
\textit{Boolean realizer} of $P$ if for every pair $(x,y)$ of distinct
elements of $P$, $x<y$ in $P$ if and only if $\tau(q(x,y,\cgB))=1$.
Ne{\v{s}}et{\v{r}}il and Pudl{\'a}k~\cite{bib:NesPud89} 
defined the \textit{Boolean dimension of $P$}, 
denoted $\bdim(P)$, as the least positive integer $d$ for which 
$P$ has a Boolean realizer $(\cgB,\tau)$ with
$|\cgB|=d$. Clearly, $\bdim(P)\le\dim(P)$, since if $\cgR=\{L_1,L_2,\dots,
L_d\}$ is a realizer of $P$, we simply take $\cgB=\cgR$ and $\tau$ as the function which
maps $(1,1,\dots,1)$ to $1$ while all other bit strings of length~$d$
are mapped to~$0$.

Trivially, $\bdim(P)=1$ if and only if $P$ is either a chain or an antichain.
We say that a poset $Q$ is a \emph{subposet} of $P$ if for any two elements $x,y$ in $Q$ 
we have $x\leqslant y$ in $Q$ if and only if $x\leqslant y$ in $P$.
So, if $Q$ is a subposet of $P$, then $\bdim(Q)\le\bdim(P)$.
In this paper, we will denote the \textit{dual} of poset $P$ as $P^*$, and we 
define it as follows: the set of vertices of $P$ and $P^*$ are the same and 
$x \leqslant y$ in $P^*$ if and only if $y \leqslant x$ in $P^*$.
Clearly, $\bdim(P)=\bdim(P^*)$.
It is an easy exercise to show that if $\bdim(P)=2$, then $\dim(P)=2$.
In~\cite{bib:TroWal17}, Trotter and Walczak prove the modestly more 
challenging fact that if $\bdim(P)=3$, then $\dim(P)=3$.
As we will see shortly, for every $d\ge4$, there is a poset $P$ with 
$\bdim(P)=4$ and $\dim(P)=d$.

Here is the setting for local dimension. Let $P$ be a poset. A \textit{partial linear extension}, abbreviated $\ple$, 
of $P$ is a linear extension of a subposet of $P$.
Whenever $\cgL$ is a family of $\ple$'s of $P$ and $u\in P$, we set 
$\mu(u,\cgL)=|\{L\in\cgL:u\in L\}|$.
In turn, we set $\mu(P,\cgL)=\max\{\mu(u,\cgL):u\in P\}$.
A non-empty family $\cgL$ of $\ple$'s of a poset $P$ is called a \textit{local 
realizer} of $P$ if the following two conditions are satisfied: (1)~If $x<y$ in $P$, there is some $L\in\cgL$ for which $x<y$ in $L$; (2)~if $(x,y) \in \Inc(P)$, there is some $L\in\cgL$ for which $x>y$ in $L$.
The \textit{local dimension of $P$}, denoted $\ldim(P)$, is then 
defined to be the least positive integer $d$ for which $P$ has a local realizer $\cgL$ with $\mu(P,\cgL)=d$.

The concept of local dimension is due to Torsten Ueckerdt who shared his ideas with participants of the workshop on \textit{Order and Geometry} (held in Gu{\l}towy, Poland, September 14--17, 2016). Clearly, this new notion resonated with participants at the workshop, and it served to stimulate renewed interest in Boolean dimension as well.

Trivially, $\ldim(P)\le\dim(P)$ for all posets $P$. Also, $\ldim(P)=1$ if and 
only if $P$ is a chain; $\ldim(Q)\le\ldim(P)$ if $Q$ is a subposet of $P$; and 
if $P^*$ is the dual of $P$, then $\ldim(P^*)=\ldim(P)$.
It is an easy exercise to show that if $\ldim(P)=2$, then $\dim(P)=2$.
However, for every $d\ge3$, there is a poset $P$ with $\ldim(P)=3$ and 
$\dim(P)=d$.

The principal result of this paper involves a construction for a family
of posets for which local dimension is unbounded.  The implications
of our construction fall into four distinct categories:

\subsection{Planar Posets}
For two elements $x,y$ of a poset $P$ we say that $y$ \emph{covers} $x$ if $x<y$ 
and there is no element $z$ with $x<z<y$.
The \emph{cover graph} of~$P$ is the undirected graph $G(P)$ whose set of vertices is the same as the set of elements of~$P$, in which $x$ is adjacent to $y$ if and only if $x$ covers $y$ or $y$ covers $x$ in~$P$.
A drawing on the plane of the cover graph of a poset $P$ is called an \emph{order 
diagram} if for any two comparable elements $x < y$ in $P$, the point in 
the plane corresponding to the element $y$ is higher than the point 
corresponding to the element $x$.
A poset $P$ is \textit{planar} if its order diagram can be drawn in the
plane without edge crossings. If a poset is planar, then its cover graph
is planar, although the converse does not hold in general.

For an integer $n\ge2$, the \textit{standard example} $S_n$ is
a height~$2$ poset with minimal elements $A=\{a_1,a_2,\dots,a_n\}$ and
maximal elements $B=\{b_1,b_2,\dots,b_n\}$. Furthermore, $a_i<b_j$ in $S_n$
if and only if $i\neq j$ (see Figure \ref{fig:StandardExample}).
\begin{figure}[hbt]%
\begin{center}%
\tikzstyle{vertex}=[circle, line width=0.3mm, draw, fill=white, inner sep=0pt, minimum width=2.2mm]%
\tikzstyle{edge}=[draw,line width=0.3mm,-]%
\begin{tikzpicture}[scale=1]
\def\n{6} % size of Standard Example
\foreach \i in {1,...,\n} {
 \coordinate (a\i) at ($(\i,0)$);
 \coordinate (b\i) at ($(\i,1.5)$); }
\foreach \i in {1,...,\n} {
 \foreach \j in {1,...,\n} {
 \ifthenelse {\i=\j} {} {
 \path[edge] (a\i) -- (b\j); } } }
\foreach \i in {1,...,\n} {
 \draw (a\i) node[vertex] {} node[below=0.15] {$a_{\i}$};
 \draw (b\i) node[vertex] {} node[above=0.15] {$b_{\i}$}; }
\end{tikzpicture}%
\end{center}%
\caption{The Standard Example}%
\label{fig:StandardExample}%
\end{figure}
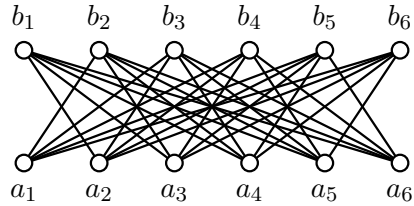%
Dushnik and Miller \cite{bib:DusMil41} showed that $\dim(S_n)=n$, for all 
$n\ge2$ (see \cite{bib:Trot-Book92}). Furthermore, it is an easy exercise to show that $S_n$ is planar 
when $2\le n\le 4$ and non-planar when $n\ge5$.

In Figure~\ref{fig:Kelly},
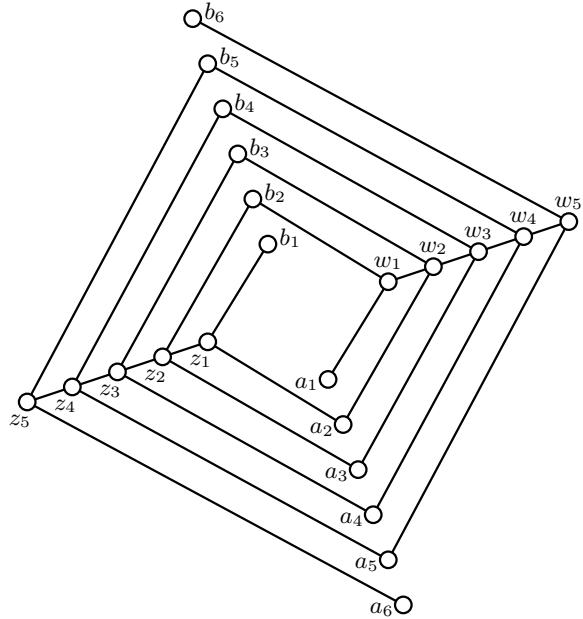
\begin{figure}[hbt]%
\begin{center}%
\tikzstyle{vertex}=[circle, line width=0.3mm, draw, fill=white, inner sep=0pt, minimum width=2.2mm]%
\tikzstyle{edge}=[draw,line width=0.3mm,-]%
\begin{footnotesize}%
\begin{tikzpicture}[scale=1]
\def\n{6} % size of Kelly Example
\pgfmathparse{\n-1}
\edef\m{\pgfmathresult}
\pgfmathparse{\n-2}
\edef\k{\pgfmathresult}
\foreach \i in {1,...,\n} {
 \coordinate (a\i) at ($(\i*0.2+0.2,-\i*0.6-0.3)$);
 \coordinate (b\i) at ($(-\i*0.2-0.2,\i*0.6+0.3)$); }
\foreach \i in {1,...,\m} {
 \coordinate (w\i) at ($(\i*0.6+0.6,\i*0.2+0.2)$);
 \coordinate (z\i) at ($(-\i*0.6-0.6,-\i*0.2-0.2)$); }
\foreach \i in {1,...,\m} {
 \pgfmathparse{\i+1}
 \edef\j{\pgfmathresult}
 \path[edge] (a\i) -- (w\i);
 \path[edge] (w\i) -- (b\j);
 \path[edge] (b\i) -- (z\i);
 \path[edge] (z\i) -- (a\j); }
\foreach \i in {1,...,\k} {
 \pgfmathparse{\i+1}
 \edef\j{\pgfmathresult}
 \path[edge] (w\i) -- (w\j);
 \path[edge] (z\i) -- (z\j); }
\foreach \i in {1,...,\n} {
 \draw (a\i) node[vertex] {} node[below left=-0.14 and 0.01] {$a_{\i}$};
 \draw (b\i) node[vertex] {} node[above right=-0.16 and 0.04] {$b_{\i}$}; }
\foreach \i in {1,...,\m} {
 \draw (w\i) node[vertex] {} node[above=0.05] {$w_{\i}$}; 
 \draw (z\i) node[vertex] {} node[below left=0.05 and -0.17] {$z_{\i}$}; }
\end{tikzpicture}%
\end{footnotesize}%
\end{center}%
\caption{The Kelly Construction}%
\label{fig:Kelly}%
\end{figure}%
we present a construction due to Kelly~\cite{bib:Kell81},
showing that for all $n\ge5$, the non-planar poset $S_n$ is a subposet
of a planar poset $K_n$. This specific figure is a diagram for
$K_6$, but it should be clear how we intend the diagram to be
depicted for other values of $n$.

 The following argument shows that the bound $\bdim(K_n)\le4$ holds for all $n\ge3$.
First, notice that $\dim(K_3)=3$ implies $\bdim(K_3)=3$. So, assume that $n\ge4$.
Then for $K_n$ we form linear orders $L_1$:
\[
a_1<w_1<a_2<\dots<w_{n-1}<a_n \ \ < \ \ b_n<z_{n-1}<b_{n-1}<\dots<z_1<b_1
\]
and $L_2$:
\[
a_n<z_{n-1}<a_{n-1}<\dots<z_1<a_1 \ \ < \ \ b_1<w_1<b_2<\ldots<w_{n-1}<b_n.
\]
Note that $L_1$ and $L_2$ are linear extensions of $K_n$.
Then form linear orders (they are \textit{not} linear extensions)
$L_3$:
\[
a_1 \! < \! b_1 \ < \ a_2 \! < \! b_2 \ < \ \ldots \ < \ a_n \! < \! b_n \ < \ 
w_1 \! < \! \ldots \! < \! w_{n-1} \! < \! z_1 \! < \! \ldots \! < \! z_{n-1}
\]
and $L_4$:
\[
z_{n-1} \! < \! \ldots \! < \! z_1 \! < \! w_{n-1} \! < \! \ldots \! < \! w_1 \ < \ 
a_n \! < \! b_n \ < \ a_{n-1} \! < \! b_{n-1} \ < \ \ldots \ < \ a_1 \! < \! b_1.
\]
Then set $\cgB=\{L_1,L_2,L_3,L_4\}$, and let $\tau:\mathbf{2}^4
\rightarrow\{0,1\}$ be defined by setting
$\tau(1,1,0,1)=\tau(1,1,1,0)=1$. The map $\tau$ sends all other
strings to $0$. It is easy to check that $(\cgB,\tau)$ is a Boolean realizer for
$K_n$.

In~\cite{bib:NesPud89}, Ne\v{s}et\v{r}il and Pudl\'{a}k remarked that
the posets in the Kelly construction have Boolean dimension at most~$4$,
and they asked if the Boolean dimension of planar
posets is bounded. It is clear from their presentation that they
believed the answer should be ``yes''.  However, this challenging
question has remained open for nearly $30$ years.

For local dimension, Ueckerdt~\cite{bib:Ueck16} noted that
$\ldim(S_n)\le 3$ for all $n\ge2$.  In fact, $\ldim(K_2)=2$ and
$\ldim(K_n)=3$ for all $n\ge3$. Here's why.
Suppose $n\ge3$.
Let $L_1$ and $L_2$ be the linear extensions of $K_n$ defined
just above. Then for each $i=1,2,\dots,n$, let $M_i$ be the
$\ple$ whose ground set is $\{a_i,b_i\}$ with $a_i>b_i$ in $M_i$.
Clearly, $\cgL=\{L_1,L_2\}\cup\{M_i:1\le i\le n\}$ is a local
realizer for $K_n$ and $\mu(z,\cgL)\le 3$ for all $z\in K_n$.

In view of these observations, it is also natural to ask whether
local dimension is bounded for the class of planar posets. Our
construction will show that the answer is ``no''.

\subsection{Components and Blocks}

We refer the reader to \cite{bib:D18} for the concepts of
graph theory, including the following terms: connected and disconnected
graphs; components; cut vertices; and $k$-connected graphs for an
integer $k\ge2$.
A block is an inclusion-maximal $2$-connected subgraph.

Here are the analogous concepts for posets.  A poset $P$ is said to
be \emph{connected} if its cover graph is connected.  A subposet $B$ of $P$ is
said to be \emph{convex} if $y\in B$ whenever $x,z\in B$ and
$x<y<z$ in $P$.   Note that when $B$ is a convex subposet of $P$, the
cover graph of $B$ is an induced subgraph of the cover graph of $P$.
A convex subposet $B$ of $P$ is called a \emph{component} of $P$ when the
cover graph of $B$ is a component of the cover graph of $P$.
A convex subposet $B$ of $P$ is called a \emph{block} of $P$, when
the cover graph of $B$ is a block in the cover graph of $P$.

When $P$ is a disconnected poset with components $C_1, C_2, 
\dots,C_t$, for some $t\ge2$, then
$\dim(P) = \max \{ 2, \max \{ \dim(C_i) : 1 \le i \le t\} \}$.
Readers may note that the preceding observation is just a special case of the 
formula for the dimension of a \emph{lexicographic sum} (see  page~23 
in~\cite{bib:Trot-Book92}).
For the local dimension, it is an easy exercise to show that
$\ldim(P)\le 2+\max\{\ldim(C_i):1\le i\le t\}$, but we do not know
whether this inequality is best possible.

The corresponding result for Boolean dimension is more complicated
and is due to
M\'{e}sz\'{a}ros, Micek and Trotter~\cite{bib:MeMiTr17}.

\begin{theorem}\label{thm:Bdim-components}
Let $P$ be a disconnected poset with components
$C_1,C_2,\dots,C_t$, for some $t\ge2$.  If $d=\max\{\bdim(C_i):
1\le i\le t\}$, then $\bdim(P)=O(2^d)$.
\end{theorem}

The inequality in Theorem~\ref{thm:Bdim-components}
cannot be improved dramatically, since it is shown in~\cite{bib:MeMiTr17}
that for large $d$, there is a disconnected poset $P$ with
$\bdim(P)=\Omega(2^d/d)$ and $\bdim(C)\le d$ for every
component $C$ of $P$.

The situation with blocks is more
complex, even for Dushnik-Miller dimension.
In~\cite{bib:TrWaWa17}, Trotter, Walczak and Wang prove the
following result for Dushnik-Miller dimension.

\begin{theorem}\label{thm:dim-blocks}
If $d\ge1$ and $\dim(B)\le d$ for every block of
a poset $P$, then $\dim(P)\le d+2$.
Furthermore, this inequality is best possible.
\end{theorem}

Boolean dimension behaves somewhat like Dushnik-Miller 
dimension with respect to blocks, as the following
inequality is also proved in~\cite{bib:MeMiTr17}.

\begin{theorem}\label{thm:Bdim-blocks}
If $d\ge1$ and $\bdim(B)\le d$ for every block $B$ of a poset $P$,
then $\bdim(P)=O(2^d)$.
\end{theorem}

Again, this inequality cannot be improved dramatically,
as it is shown in~\cite{bib:MeMiTr17}
that for large $d$, there is a poset $P$ with
$\bdim(P)=\Omega(2^d/d)$ and $\bdim(B)\le d$ for every
block $B$ of $P$.

Our construction will show that local dimension behaves
quite differently with respect to blocks.
We will prove:

\begin{theorem}\label{thm:ldim-blocks}
For every $d\ge 1$, there
is a poset $P$ such that $\ldim(P)=d$ while $\ldim(B)\le 3$
for every block $B$ of $P$.
\end{theorem}

\subsection{Structural Graph Theory}

The first major result linking dimension with structural graph
theory is due to Joret, Micek, Milans, Trotter, Walczak and 
Wang~\cite{bib:JMMTWW16},
who showed that for each pair $(t,h)$ of positive integers, there is a
least positive integer $d(t,h)$ so that if $P$ is a poset of height
$h$ and the tree-width\footnote{We refer the reader to \cite{bib:D18} for the concepts of tree-width and path-width.} of the cover graph of $P$ is~$t$, then
$\dim(P)\le d(t,h)$. A poset of height~$1$ is an antichain and has
dimension at most~$2$, so it is of interest to study $d(t,h)$ only
when $h\ge2$.  Trotter and Moore~\cite{bib:TroMoo77} showed that
$d(1,h)=3$ for all $h\ge2$, and Joret, Micek, Trotter, Wang, and
Wiechert~\cite{bib:JMTWW16} showed that $d(2,h)\le 1276$ for all $h\ge2$.
Recently, Seweryn~\cite{bib:Sewe19} has given the following substantive 
improvement:\quad $d(2,h)\le 12$.
It is a relatively simple exercise to show that the posets in the Kelly 
construction~\cite{bib:Kell81} have cover graphs with path-width at most~$3$, 
so $d(t,h)$ goes to infinity with $h$ when $t\ge3$.
The best bounds to date in the general case are due to Joret, Micek, Ossona de 
Mendez and Wiechert~\cite{bib:JMOW16}:

\begin{equation}
2^{\Omega(h^{\lfloor (t-1)/2 \rfloor})}
 \leq d(t,h)
 \leq 4^{\binom{t+3h-3}{t}}.
\end{equation}

However, Felsner, M\'{e}sz\'{a}ros and Micek~\cite{bib:FeMeMi17} proved that
the Boolean dimension of a poset is bounded
in terms of the tree-width of its cover graph, independent of its
height.  Formally, here is their result.

\begin{theorem}\label{thm:bdim-tw}
For every $t\ge1$, there is a least positive integer $d(t)$ so that
if $P$ is a poset whose cover graph has tree-width~$t$, then
$\bdim(P)\le d(t)$.
\end{theorem}

Barrera-Cruz, Prag, Smith, Taylor and Trotter~\cite{bib:BPSTT17} proved that 
the local dimension of a poset is bounded
in terms of the path-width of its cover graph, independent of its
height.  Formally, here is their result.

\begin{theorem}\label{thm:ldim-pw}
For every $t\ge1$, there is a least positive integer $d'(t)$ so that
if $P$ is a poset whose cover graph has path-width~$t$, then
$\ldim(P)\le d'(t)$.
\end{theorem}

However, it is also shown in~\cite{bib:BPSTT17} that the analogue of
Theorem~\ref{thm:bdim-tw} for local dimension is false:

\begin{theorem}\label{thm:ldim-tw}
For every $d\ge1$, there exists a poset $P$ with $\ldim(P)>d$ such
that the cover graph of $P$ has tree-width at most~$3$.
\end{theorem}

Our construction provides an alternative proof of 
Theorem~\ref{thm:ldim-tw}.

\subsection{Bounded Boolean Dimension and Unbounded Local Dimension}

Trotter and Walczak~\cite{bib:TroWal17} proved that if $P$ is a poset and 
$\ldim(P)\le3$, then $\bdim(P)\le 8442$.
However, for each $d\ge4$, they also proved that there is a poset $P$ with 
$\ldim(P)=4$ and $\bdim(P)=d$.
They also showed that if $\bdim(P)=3$, then $\ldim(P)=\dim(P)=3$.
In~\cite{bib:BPSTT17}, it is shown that for each $d\ge1$, there is a poset $P$ 
with $\bdim(P)\le 4$ and $\ldim(P)>d$.
Our construction provides another instance of a family of posets where Boolean 
dimension is bounded and local dimension is not.

\section{Our Construction}

For the remainder of the paper, whenever we discuss a pair $(n,d)$, it will be 
assumed that $n$ and $d$ are integers with $n\ge2$ and $d\ge1$.
Also, we do not distinguish between an element $x$ and a sequence $(x)$ of 
length~$1$.

Fix an integer $n\ge2$.
Then for each $d\ge1$, we define a planar poset $\Kelly(n,d)$ via a recursive 
process.
The elements of $\Kelly(n,d)$ will be identified with sequences of length $d$ from the set 
$\{a_1, \ldots, a_n, b_1, \ldots, b_n, w_1, \ldots, w_{n-1}, z_1, \ldots, z_{n-1}\}$. 
As suggested by the notation, the poset $\Kelly(n,1)=K_{n+1}$, i.e., 
$\Kelly(n,1)$ is just the Kelly construction illustrated in 
Figure~\ref{fig:Kelly}.

Now suppose that we have defined the planar poset $\Kelly(n,d)$ for some \mbox{$d\ge1$}.
Suppose further that we have a planar drawing without crossings of the 
order diagram of $\Kelly(n,d)$ and that in this drawing, 
$(b_n)$ is the highest point. To form $\Kelly(n,d+1)$, we take the 
drawing of $\Kelly(n,1)$ illustrated in Figure~\ref{fig:Kelly} and make the following changes: 

For each $i=1,2,\dots,n$, we take a suitably small scaling of 
the drawing of $\Kelly(n,d)$ and identify the point $(b_n)$ in 
$\Kelly(n,d)$ with the point $a_i$ in $\Kelly(n,1)$.  We change the 
label of the points in the copy of $\Kelly(n,d)$ by prepending the 
symbol $a_i$ at the start of the sequence.

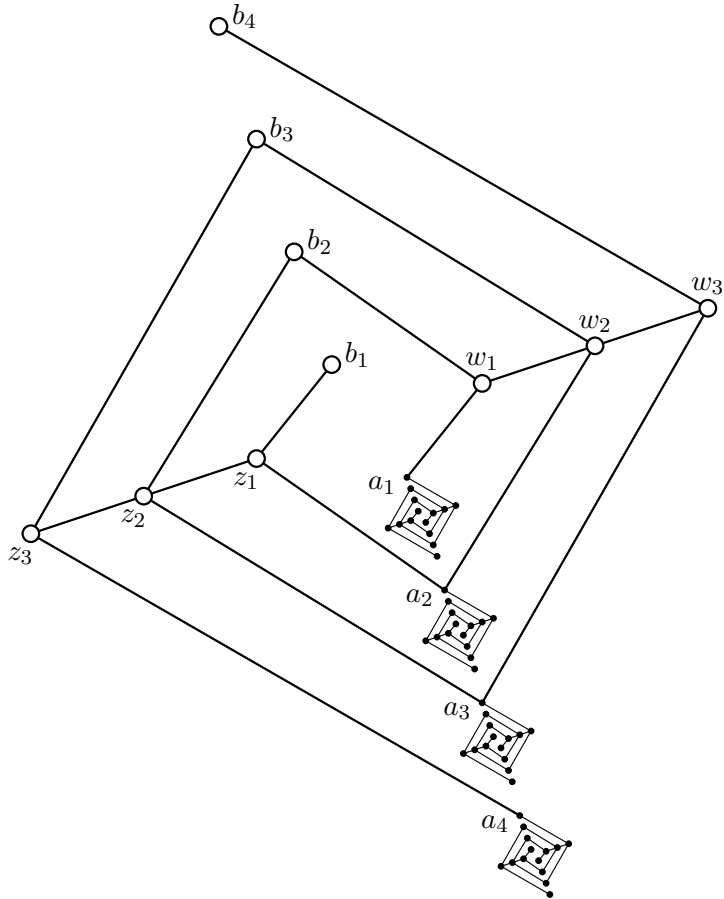
\begin{figure}[hbt]%
  \begin{center}%
  \tikzstyle{svertex}=[circle, line width=0.3mm, draw, fill=black, inner sep=0pt, minimum width=0.6mm]%
  \tikzstyle{vertex}=[circle, line width=0.3mm, draw, fill=white, inner sep=0pt, minimum width=2.2mm]%
  \tikzstyle{edge}=[draw,line width=0.3mm,-]%
  \tikzstyle{sedge}=[draw,line width=0.15mm,-]%
  
  \newcommand{\smallKellyMin}[4]{
  % \draw (#1) node[vertex] {$#1$};
  
  \def\nn{#1} % size of Kelly Example
  \def\ddist{#2}
  \pgfmathparse{\nn-1}
  \edef\mm{\pgfmathresult}
  \pgfmathparse{\nn-2}
  \edef\kk{\pgfmathresult}
  
  \pgfmathparse{\nn*\ddist/3+#3}
  \edef\xRel{\pgfmathresult}
  \pgfmathparse{-\nn*\ddist+\ddist/2+#4}
  \edef\yRel{\pgfmathresult}
  
  \foreach \ii in {1,...,\nn} {
   \coordinate (aa\ii) at ($(\ii*\ddist/3+\xRel,-\ii*\ddist+\ddist/2+\yRel)$);
   \coordinate (bb\ii) at ($(-\ii*\ddist/3+\xRel,\ii*\ddist-\ddist/2+\yRel)$); }
  \foreach \ii in {1,...,\mm} {
   \coordinate (ww\ii) at ($(\ii*\ddist+\xRel,\ii*\ddist/3+\yRel)$);
   \coordinate (zz\ii) at ($(-\ii*\ddist+\xRel,-\ii*\ddist/3+\yRel)$); }
  \foreach \i in {1,...,\mm} {
   \pgfmathparse{\i+1}
   \edef\j{\pgfmathresult}
   \path[sedge] (aa\i) -- (ww\i);
   \path[sedge] (ww\i) -- (bb\j);
   \path[sedge] (bb\i) -- (zz\i);
   \path[sedge] (zz\i) -- (aa\j); }
  \foreach \i in {1,...,\kk} {
   \pgfmathparse{\i+1}
   \edef\j{\pgfmathresult}
   \path[sedge] (ww\i) -- (ww\j);
   \path[sedge] (zz\i) -- (zz\j); }
  \foreach \i in {1,...,\nn} {
   \draw (aa\i) node[svertex] {};
   \draw (bb\i) node[svertex] {}; }
  \foreach \i in {1,...,\mm} {
   \draw (ww\i) node[svertex] {}; 
   \draw (zz\i) node[svertex] {}; }
  }

  \newcommand{\smallKellyMax}[4]{
  % \draw (#1) node[vertex] {$#1$};
  
  \def\nn{#1} % size of Kelly Example
  \def\ddist{#2}
  \pgfmathparse{\nn-1}
  \edef\mm{\pgfmathresult}
  \pgfmathparse{\nn-2}
  \edef\kk{\pgfmathresult}
  
  \pgfmathparse{-\nn*\ddist/3+#3}
  \edef\xRel{\pgfmathresult}
  \pgfmathparse{\nn*\ddist-\ddist/2+#4}
  \edef\yRel{\pgfmathresult}
  
  \foreach \ii in {1,...,\nn} {
   \coordinate (aa\ii) at ($(\ii*\ddist/3+\xRel,-\ii*\ddist+\ddist/2+\yRel)$);
   \coordinate (bb\ii) at ($(-\ii*\ddist/3+\xRel,\ii*\ddist-\ddist/2+\yRel)$); }
  \foreach \ii in {1,...,\mm} {
   \coordinate (ww\ii) at ($(\ii*\ddist+\xRel,\ii*\ddist/3+\yRel)$);
   \coordinate (zz\ii) at ($(-\ii*\ddist+\xRel,-\ii*\ddist/3+\yRel)$); }
  \foreach \i in {1,...,\mm} {
   \pgfmathparse{\i+1}
   \edef\j{\pgfmathresult}
   \path[sedge] (aa\i) -- (ww\i);
   \path[sedge] (ww\i) -- (bb\j);
   \path[sedge] (bb\i) -- (zz\i);
   \path[sedge] (zz\i) -- (aa\j); }
  \foreach \i in {1,...,\kk} {
   \pgfmathparse{\i+1}
   \edef\j{\pgfmathresult}
   \path[sedge] (ww\i) -- (ww\j);
   \path[sedge] (zz\i) -- (zz\j); }
  \foreach \i in {1,...,\nn} {
   \draw (aa\i) node[svertex] {};
   \draw (bb\i) node[svertex] {}; }
  \foreach \i in {1,...,\mm} {
   \draw (ww\i) node[svertex] {}; 
   \draw (zz\i) node[svertex] {}; }
  }
  
  %\begin{footnotesize}%
  \begin{tikzpicture}[scale=1]
  \def\n{4} % size of Kelly Example
  \def\dist{1.5}
  \pgfmathparse{\n-1}
  \edef\m{\pgfmathresult}
  \pgfmathparse{\n-2}
  \edef\k{\pgfmathresult}
  \foreach \i in {1,...,\n} {
   \coordinate (a\i) at ($(\i*\dist/3,-\i*\dist+\dist/2)$);
   \coordinate (b\i) at ($(-\i*\dist/3,\i*\dist-\dist/2)$);
  }
  \foreach \i in {1,...,\m} {
    \smallKellyMin{4}{0.15}{\i*\dist/3}{-\i*\dist+\dist/2}
    \coordinate (w\i) at ($(\i*\dist,\i*\dist/3)$);
   \coordinate (z\i) at ($(-\i*\dist,-\i*\dist/3)$); }
  \foreach \i in {1,...,\m} {
   \pgfmathparse{\i+1}
   \edef\j{\pgfmathresult}
   \path[edge] (a\i) -- (w\i);
   \path[edge] (w\i) -- (b\j);
   \path[edge] (b\i) -- (z\i);
   \path[edge] (z\i) -- (a\j); }
  \foreach \i in {1,...,\k} {
   \pgfmathparse{\i+1}
   \edef\j{\pgfmathresult}
   \path[edge] (w\i) -- (w\j);
   \path[edge] (z\i) -- (z\j); }
  \foreach \i in {1,...,\n} {
   \draw (b\i) node[vertex] {} node[above right=-0.16 and 0.04] {$b_{\i}$}; }
  \foreach \i in {1,...,\m} {
    \draw (a\i) node[] {} node[below left=-0.14 and 0.01] {$a_{\i}$};
    \draw (w\i) node[vertex] {} node[above=0.05] {$w_{\i}$}; 
   \draw (z\i) node[vertex] {} node[below left=0.05 and -0.17] {$z_{\i}$}; }
   \draw (a\n) node[vertex] {} node[below left=-0.14 and 0.01] {$a_{\n}$};
  \end{tikzpicture}%
  %\end{footnotesize}%
  \end{center}%
  \caption{Our construction for $n=3$ and $d=2$}%
  %\label{fig:K42popr}%
  \end{figure}%

With the obvious requirements regarding scaling in mind,
it is clear that posets in the family $\mathbb{K}=\{\Kelly(n,d):n\ge2,d\ge1\}$
are planar. We note that for each pair $(n,d)$, a block of the poset
$\Kelly(n,d)$ is isomorphic to a subposet of the Kelly construction $K_{n+1}$.
Accordingly, if $B$ is a block in $\Kelly(n,d)$, then
$\ldim(B)\le 3$.

We also note that posets in the family $\mathbb{K}$ have
tree-width at most~$3$. Here's why: It is easy to see that the cover graph of the poset $\Kelly(n,1)$ has path-width at most~$3$, so it has tree-width
at most~$3$. One of the basic properties of tree-width is that the
tree-width of a connected graph is just the maximum tree-width of its
blocks. Since the blocks of $\Kelly(n,d)$ are isomorphic to subposets
of the Kelly construction $K_{n+1}$, they have tree-width at most~$3$.
However, it should be noted that the path-width of cover graphs 
in the family $\mathbb{K}$ is \textit{not} bounded.

Let $A=\{a_1,a_2,\dots,a_n\}$ and $B=\{b_1,b_2,\dots,b_n\}$.
The elements of $\Kelly(n,d)$ will be called the \textit{core points} if they are identified with
sequences from $\big(\bigcup_{k=1}^dA^{k-1}\times B\big)\cup A^d$, and we will 
let $\Core(n,d)$ denote the subposet of $\Kelly(n,d)$ determined by the core 
points. Roughly speaking, the core points are those sequences that consist of 
$d$ elements of $A$, or less than $d$ elements of $A$ followed by exactly one 
element of $B$.
A subtle detail, that is important later on, is that we are not using 
$a_{n+1}$ and $b_{n+1}$ in the definition of the $\Core(n,d)$.
Note that $\Core(n,d)$ will always be a proper subposet of $\Kelly(n,d)$.
Also, note that $\Core(n,1)$ is just the standard example $S_n$.

For the arguments to follow, it is important to 
understand the structure of the subposet $\Core(n,d)$.
We state the following elementary proposition for the emphasis.

\begin{proposition}\label{pro:coreNew}
Let $u$ and $v$ be two distinct elements of $\Core(n,d)$. Then 
$u<v$ in $\Core(n,d)$ if and only if there is $k\in[d]$ such 
that $u=(a_{i_1},\dots,a_{i_{k-1}},a_{i_k},\dots,a_{i_d})\in 
A^d$ and $v=(a_{i_1},\dots,a_{i_{k-1}},b_j)\in A^{k-1}\times B$, 
where $a_{i_k}<b_j$ in $S_n$.
\end{proposition}

For example, when $n=8$ and $d=6$, we have
\[
u=(a_2,a_6,a_7,a_5,a_5,a_4)<(a_2,a_6,b_3)=v
\]
in $\Kelly(8,6)$. Note that $a_7<b_3$ in $S_8$.

We also note that the subposet $\Core(n,d)$ has height~$2$.
Furthermore, the minimal elements of $\Core(n,d)$ are those 
sequences of length $d$ with all coordinates in $A$, while the 
maximal elements of $\Core(n,d)$ are those sequences of length 
at most $d$ with all coordinates except the last in $A$. 

\section{The Local Dimension of Posets in the Class \texorpdfstring{$\mathbb{K}$}{K}}

In this section, we will show that the local dimension of
posets in the class $\mathbb{K}$ is unbounded.  This shows:
(1)~local dimension is not bounded for the class of 
planar posets; (2)~the local
dimension of a poset is not bounded in terms of the maximum local 
dimension of its blocks; and (3)~the local dimension of a poset
cannot be bounded in terms of the tree-width of its cover graph,
independently of its height. Finally, we will have given another
example of a family of posets where Boolean dimension is 
bounded and local dimension is not.

The proof of our main theorem will
require a special case of a result which has become known as 
the ``Product Ramsey Theorem'', appearing in the classic 
text~\cite{bib:GrRoSp90} as Theorem~5 on page~113. However, we will 
use slightly different notation in discussing
this result.

Given a finite set $X$ and
an integer $k$ with $0\le k\le|X|$, we denote the
set of all $k$-element subsets of $X$ by $\binom{X}{k}$.
When $T_1,T_2,\dots,T_t$ are $k$-element subsets of
$X_1,X_2,\dots,X_t$, respectively, we refer to
the product $g=T_1\times T_2\times\dots\times T_t$ as
a \emph{$\bfk^t$-grid} in $X_1\times X_2\times\dots\times X_t$.

Here is a formal statement of the version of the Product
Ramsey Theorem we will use in our argument.

\begin{theorem}\label{thm:product-ramsey}
For every $4$-tuple $(r,t,k,m)$ of positive integers with $m\ge k$, there is 
a~least positive integer $n_0=\PRam(r,t,k,m)$ with $n_0\ge m$ such that if 
$|X_i|\ge n_0$ for every $i=1,2,\dots,t$, then whenever we have a~coloring 
$\phi$ which assigns to each $\bfk^t$-grid $g$ in $X_1 \times X_2 \times \dots 
\times X_t$ a color $\phi(g)$ from a~set~$R$ of $r$ colors, then there is 
a~color $\alpha \in R$ so that for each $j=1,2,\dots,t$, there is 
an~$m$-element subset $H_j\subseteq X_j$ such that $\phi(g) = \alpha$ for 
every $\bfk^t$-grid $g$ in $H_1\times H_2\times\dots\times H_t$.
\end{theorem}

Actually, we will only use the case where $k=1$, and now the
theorem becomes a multi-dimensional version of the pigeon-hole
principle. Readers who would be interested in how this
theorem is applied to combinatorial problems on posets when
$k\ge2$ are encouraged to consult~\cite{bib:BPSTT17}, 
\cite{bib:FeFiTr99} and~\cite{bib:TrWaWa17}.

Now we are ready to state and prove our main theorem.

\begin{theorem}\label{thm:ldim-planar}
For every $d\ge1$, there exists a least positive integer $n_d$ with
$n_d\ge2$ so that if $n\ge n_d$, then $\ldim(\Kelly(n,d))\ge d$.
\end{theorem}

\begin{proof}
The theorem holds trivially for $d\le 2$, with $n_1=n_2=2$. So for the
balance of the argument, we fix a value of $d$ with $d\ge 3$. 
Since $\Core(n,d)$ is a subposet of $\Kelly(n,d)$, it suffices
to show that $\ldim(\Core(n,d))\ge d$, provided $n$ is sufficiently
large. The argument will proceed by contradiction, i.e., we will
assume that $\cgL$ is a local realizer for $\Core(n,d)$ with
$\mu(z,\cgL)\le d-1$ for every $z\in \Core(n,d)$. We will then show
that this leads to a contradiction, provided $n$ is sufficiently large.

We will now describe a recursive procedure which consists of $d$ steps.
The procedure will utilize a rapidly growing sequence $(p_1,p_2,\dots,p_d)$ of integers with $p_1=d$.
For the moment, we defer the explanation as to how  this sequence is determined, but in time it will be clear that this is done by repeated applications of Theorem~\ref{thm:product-ramsey}. 

The key idea of the proof is presented in the following technical claim for 
which we need to define $A(H)=\{a_h\in A:h\in H\}$ for each $H\subseteq[n]$ 
and $\Min(X)$ as the set of all minimal elements in $X$.

\begin{claim}\label{clm:ind}
Let $m\in[d]$ and $p\in\mathbb{N}$. 
Then there exists $n\in\mathbb{N}$ such that the following statements hold:
\begin{enumerate}
\item There are $m-1$ indices $h_1,\ldots,h_{m-1}\in [n]$.
\item There exist $d+1-m$ sets $H_m,\dots,H_d \subseteq [n]$, with $|H_m|=\dots=|H_d|=p$.
\item If $\mathcal{L}$ is a local realizer for $\Core(n,d)$, then there exists an $(m-1)$-element family $\cgM \subseteq\cgL$  of $\ple$'s such that each element in
$$
\{a_{h_1}\}\times\dots\times\{a_{h_{m-1}}\}\times A(H_m) 
\times \dots \times A(H_d) \subseteq \Min(\Core(n,d))
$$
appears in every $\ple$ from $\cgM$.
\end{enumerate}

\end{claim}

It is worth mentioning that for each $i\in[m]$ the value $p_i$ equals $p$ from 
the above claim for $m=d+1-i$ and can be calculated by analyzing the~sequence 
of the calls of the claim in reverse order $m = d, \ldots, d+1-i$.

\begin{proof}% [Proof of Claim \ref{clm:ind}]
We prove the claim by induction on $m\in[d]$.
For $m=1$, we must be concerned only about item~(2) in this list by 
setting $n=p$ and $H_i=[n]$ for each $i=1,\dots,d$.

Here's how the inductive step is carried out.
We will prove the statement of the claim for $m+1$ and any $p'\in\mathbb{N}$.
Let us assume that it is true for $m$ and for some fixed $p\in\mathbb{N}$ which depends on $p'$. 
(How big $p$ must be will be described later.)
In consequence we obtain an integer $n$, a set $\cgM$ of ple's, appropriate $h_1,\ldots,h_{m-1}\in[n]$, and $H_m,\ldots,H_d\subseteq[n]$.
For each $i=m,\dots,d$, let $H_i=H_{i,1}\cup H_{i,2}\cup\dots\cup H_{i,m}$ be 
a partition into disjoint subsets with $|H_{i,j}|\ge \lfloor p/m \rfloor$ for 
all $j=1,2,\dots,m$.
For each $j\in[m]$, we let
\begin{align*}
U_j&=\{a_{h_1}\}\times\dots\times\{a_{h_{m-1}}\}\times A(H_{m,j}) \times 
 \dots \times A(H_{d,j}),\\
V_j&=\{a_{h_1}\}\times\dots\times\{a_{h_{m-1}}\}\times B(H_{m,j}),
\end{align*}
where $B(H)=\{b_h\in B:h\in H\}$ for each $H\subseteq[n]$.
Note that $U_j \subseteq \Min(\Core(n,d))$ and $V_j\subseteq \Max(\Core(n,d))$.
Also, for each $j=1,2,\dots,m$, we let $\cgL_j$ denote the subfamily 
of $\cgL$ consisting of those $\ple$'s which reverse at least one 
incomparable min-max pair in $U_j\times V_j$.
Note that each $\cgL_j$ is non-empty whenever $p/m\ge 1$. 
We observe that if $1\le j<j'\le m$, then $\cgL_j\cap\cgL_{j'}=\emptyset$.
This follows from the fact, that if $(u,v)$ and $(u',v')$ are 
incomparable min-max pairs from $U_j\times V_j$ and $U_{j'}\times V_{j'}$, 
respectively, then $u<v'$ and $u'<v$ in $\Core(n,d)$. 

Since $|\cgM|=m-1$, it follows that there is some integer $j_0\in[m]$ 
so that $\cgM\cap\cgL_{j_0}=\emptyset$.
Let $h_m$ be any integer in $H_{m,j_0}$ and let
\begin{align*}
W & = \{a_{h_1}\}\times\dots\times\{a_{h_{m-1}}\}\times\{a_{h_{m}}\}
 \times A(H_{m+1,j_0}) \times \dots \times A(H_{d,j_0}),\\
\{v\} & = \{a_{h_1}\}\times\dots\times\{a_{h_{m-1}}\}\times\{b_{h_{m}}\}.
\end{align*}
It is worth noting that $W\subseteq U_{j_0} \subseteq \Min(\Core(n,d))$ 
and $v\in V_{j_0}  \subseteq \Max(\Core(n,d))$.
Moreover we observe that $v\parallel w$ in $\Core(n,d)$ for all $w\in W.$
It follows that for each $w\in W$, there is some $\ple$ $L$ in 
$\cgL_{j_0}\subseteq\cgL-\cgM$ with $w>v$ in $L$.
Since there are at most $d-1$ $\ple$'s in $\cgL$ in which 
$v$ appears, this results in a coloring of the elements of $W$ using 
at most $d-1$ colors.
Since there is a natural one-to-one correspondence between elements 
of $W$ and $\mathbf{1}^{d-m}$ grids in $H_{m+1,j_0}\times\dots\times 
H_{d,j_0}$, we are then in a position to apply Theorem~\ref{thm:product-ramsey}.

In particular, given a value of $p'$, we will assume that $p/m$ 
is large enough to guarantee that there is a $\ple$ $L\in\cgL_{j_0}
\subseteq\cgL-\cgM$ and a family $K_{m+1},\dots,K_d$ of sets with 
$K_i\subseteq H_{i,j_0}$ and $|K_i|=p'$ for each 
$i=m+1,m+2,\dots,d$, such that $w>v$ in $L$ for all
$$
w \in \{a_{h_1}\}\times\dots\times\{a_{h_{m}}\}\times A(K_{m+1}) 
 \times \dots \times A(K_{d}) \subseteq W.
$$
We then add the $\ple$ $L$ to $\cgM$ and set 
$H_i=K_i$ for all $i=m+1,m+2,\dots,d$.
\end{proof}

The final contradiction occurs as we invoke Claim~\ref{clm:ind} for $m=d$ and $p=d$.
Now we have integers $h_1,\ldots,h_{d-1}$, a set $H_d$ of size $d$, and a subfamily $\cgM$ of $\cgL$ with $|\cgM|=d-1$.
It follows that there are exactly $d$ incomparable min-max pairs of the form $(u,v)$ with 
\begin{align*}
u&\in\{a_{h_1}\}\times\{a_{h_2}\}\times\dots\times\{a_{h_{d-1}}\}\times A(H_d) 
 \quad\text{and}\\
v&\in\{a_{h_1}\}\times\{a_{h_2}\}\times\dots\times\{a_{h_{d-1}}\}\times B(H_d).
\end{align*}

Furthermore, these incomparable min-max pairs form a subposet of $\Core(n,d)$ which is isomorphic to the standard example $S_d$.
Because one ple can reverse at most one of the $d$ incomparable pairs, there are (at least) $d$ distinct $\ple$'s in $\cgL$ 
which reverse these pairs and at least one of these does not belong to $\cgM$.
The $d-1$ ple's from $\cgM$ with the one outside we just got are witnesses for that, there is some $u$ with $\mu(u,\cgL)\ge d$.
The contradiction completes the proof.
\end{proof}

\end{document}